\documentclass{amsart}

\usepackage{amsthm,amssymb,verbatim}
\usepackage{graphicx}
\usepackage{enumerate}

 \newcommand{\RR}{\mathbf{R}}  

\input epsf
\def\begfig {
\begin{figure}
\small }
\def\endfig {
\normalsize
\end{figure}
}

\swapnumbers

    \newtheorem{theorem}    {Theorem}       [section]
    \newtheorem{lemma}      [theorem]       {Lemma}

    \newtheorem*{theorem*}{Theorem}

    \theoremstyle{definition}
    \newtheorem{definition}  [theorem] {Definition}

    \theoremstyle{definition}
    \newtheorem{remark}   [theorem]       {Remark}
    
    \theoremstyle{definition}
    \newtheorem{remarks} [theorem]       {Remarks}

 \renewcommand{\SS}{\mathbf{S}}  

\usepackage{hyperref}
\usepackage[alphabetic]{amsrefs}

\begin{document}

 \renewcommand{\SS}{\mathbf{S}}  

\renewcommand{\thesubsection}{\thetheorem}

\title{Axial minimal surfaces in ${\SS}^2\times\RR$ are helicoidal}
\begin{abstract}
We prove that if a complete, properly embedded, finite-topology
minimal surface in $\SS^2\times \RR$ contains a line, then its ends
are asymptotic to helicoids, and that if the surface is an annulus, 
it must be a helicoid.
\end{abstract}
\author{David Hoffman}
\address{Department of Mathematics\\ Stanford University\\ Stanford, CA 94305}
\email{hoffman@math.stanford.edu}
\author{Brian White}
\thanks{The research of the second author was supported by the NSF
  under grant~DMS~0707126. }
\email{white@math.stanford.edu}
\date{October 16, 2009.}

\maketitle
\section{Introduction}
There is a rich theory of complete properly embedded minimal surfaces
of finite topology in $\RR^3$.    In particular, we now have a good understanding of the ends
of such surfaces:  aside from the plane, every such surface
either has one end, in which it case it is asymptotic to a helicoid~\cite{BernsteinBreiner2008}, or it has more than one end,
in which case each end is asymptotic to a plane or to a catenoid~\cite{Collin1997}, 
  \cite{HoffmanMeeks1989}, \cite{MeeksRosenberg1993}. 
  For the rest of this introduction, let us use ``minimal surface" to mean ``complete, properly embedded
minimal surface with finite topology". (Colding and Minicozzi  \cite{ColdingMinicozzi2008} have proved that every  complete embedded minimal surface with finite topology  in $\RR^3$ is properly embedded, so the assumption of properness is not necessary.)

It is interesting to try to classify the ends of minimal surfaces
in homogeneous  $3$-manifolds other than $\RR^3$.  This paper deals
with the ambient manifold $\SS^2\times\RR$.  
(The fundamental paper on minimal surfaces in $\SS^2\times\RR$ is Rosenberg \cite{Rosenberg2002}. 
The survey \cite{Rosenberg2003} is a good introduction to this paper 
as well as to the papers of \cite{MeeksRosenberg2005}, \cite{Hauswirth2006} and 
         \cite{PedrosaRitore}  mentioned below.)  In that case, the only compact minimal surfaces
are  horizontal $2$-spheres.  Any noncompact example has exactly two ends, both annular, one going up
and one going down.  Therefore the genus-zero, noncompact minimal surfaces in $\SS^2\times\RR$ 
are all annuli.  The minimal annuli that are foliated by horizontal circles have been classified by Hauswirth \cite{Hauswirth2006}. 
They form a two-parameter family that contains
on its boundary the helicoids (defined in Section~\ref{helicoids}) and the  
unduloids constructed by Pedrosa and Ritore \cite{PedrosaRitore}.
There are no other known minimal annuli.

 These facts suggest the following two questions posed by Rosenberg:
 \begin{enumerate}
\item Is every minimal annulus in $\SS^2\times \RR$ one of the known examples? That is, is every minimal annulus fibered by horizontal circles?
\item 
If so, must each end of any minimal surface in $\SS^2\times \RR$ be asymptotic to one of the known 
minimal annuli?
\end{enumerate}
In this paper, we show that the answer to both questions is ``yes" in case the surface
is an axial surface, i.e., in case the surface contains a vertical line. 
In particular, the axial minimal annuli in $\SS^2\times \RR$ are precisely the helicoids.

The assumption that the surface contain a line is a very strong one, but there are many minimal
surfaces that have that property.  Indeed, in \cite{HoffmanWhiteS2xR}  we prove existence of  axial examples of
every genus $g$ and every vertical flux. (See also \cite{HoffmanWhite2008}.)  
By  Theorem~\ref{maintheorem} below, those examples
are all asymptotic to helicoids, so we call them ``genus-$g$ helicoids".

Combining the results of that paper with Theorem~\ref{maintheorem}, we have:

\begin{theorem}\label{genus-g}
For every helicoid $H$ of finite pitch in $\SS^2\times \RR$ and for every genus $g>0$, there are at least two
genus-$g$ properly embedded, axial minimal surfaces whose ends are, after suitable rotations, asymptotic to $H$.
The two surfaces  are not congruent to each other by any orientation-preserving isometry of $\SS^2\times \RR$.
If $g$ is even, they are not congruent to each other by any isometry of $\SS^2\times \RR$.
\end{theorem}

The totally geodesic cylinder $\SS^1\times \RR$ may be thought of as a helicoid of infinite pitch.
In this case, the methods of~\cite{HoffmanWhiteS2xR} still produce two examples for each genus, 
but the proof that the two examples are not congruent breaks down.
Earlier, by a different method, 
Rosenberg \cite{Rosenberg2002} explicitly constructed, for each $g$, 
an axial, genus-$g$ minimal surface asymptotic 
to a cylinder.

\stepcounter{theorem}
\subsection {Helicoids} \label{helicoids}


Let $O$ and $O^*$ be a pair of antipodal points in $\SS^2\times\{0\}$ and let $Z$ and $Z^*$ be
vertical lines passing through those points. 
Let $\sigma_{\alpha, v}$ denote the screw motion of $\SS^2\times \RR$ consisting of rotation through angle $\alpha$
about the axes $Z$ and $Z^*$ followed by vertical translation by $v$.
A {\em helicoid} with axes $Z$ and $Z^*$ is a surface of the form
\[
  \bigcup_{z\in \RR} \sigma_{\kappa z, z} \, X
\]
where $X$ is a horizontal great circle that intersects $Z$ and $Z^*$.
The {\em pitch} of the helicoid is $2\pi/\kappa$, and it equals twice the vertical distance between successive sheets of the surface. Unlike the situation in $\RR^3$, helicoids of different pitch do not differ by a homothety of $\SS^2\times \RR$; there are no such homotheties.  Note that a cylinder is a helicoid with infinite pitch ($\kappa=0$), and that as  $\kappa\rightarrow \infty$  the helicoids associated with $\kappa$  converge to a minimal lamination of $\SS^2\times \RR$ by level spheres  with singular set of convergence equal to the axes $ Z \cup Z^*$.


  The main result of this paper is the following theorem:
  
   \begin{theorem}\label{maintheorem}
 Let $M$ be a properly embedded, axial minimal surface in $\SS^2\times \RR$ with bounded curvature
 and without boundary.
\begin{enumerate} 
\item If $E$ is an annular end of $M$, then $E$ is asymptotic to
  a helicoid;
\item If $M$ is an annulus, then $M$ is a helicoid;
\item  If $M$ has finite topology, then each of its two ends is asymptotic to a helicoid,
   and the two helicoids are congruent to each other by  a rotation.
\end{enumerate}
\end{theorem}

\begin{remarks}  Meeks and Rosenberg \cite{MeeksRosenberg2005} proved that  a 
properly embedded minimal surface with finite topology in $\SS^2 \times \RR$  has bounded curvature. 
Thus our assumption that the surfaces we consider have bounded curvature is always satisfied.

In statement~(1), it is not necessary that $E$ be part of a complete surface without boundary.
The statement is true (with essentially the same proof) for any 
properly embedded annulus $E\subset \SS^2\times [z_0,\infty)$ such
that $\partial E\subset \partial  \SS^2\times \{z_0\}$ and such that $E$ contains a vertical ray.

We do not know whether the two helicoids referred to in statement~(3) must be the same.  
See the discussion in Remark~\ref{displaced-ends} below.
\end{remarks}

We would like to thank Harold Rosenberg
 for helpful discusions.

\section{A convexity lemma}
\stepcounter{theorem}
\subsection{Axial surfaces are symmetric and have two axes}  
Suppose that $M$ is a properly embedded, axial minimal surface in $\SS^2\times \RR$.
Then $M$ contains a vertical line~$Z$.  We claim that $M$ must also contain the antipodal
line $Z^*$, i.e., the line consisting of all points at distance $\pi r$ from $Z$, where $r$ is the radius
of the $\SS^2$.  To see this,
let $\rho_Z: \SS^2\times \RR\to \SS^2\times \RR$ denote rotation by $\pi$ about $Z$.
By Schwarz reflection, $\rho_Z$ induces an
orientation-reversing isometry of $M$.  In particular, $\rho_Z$ interchanges the
two components of the complement of $M$.  Thus no point in the complement of 
$M$ is fixed by $\rho_Z$, so the fixed points of $\rho_Z$ must all lie in $M$.
The fixed point set of $\rho_Z$ is precisely $Z\cup Z^*$,  so $Z^*$ must lie in $M$, as claimed.

\stepcounter{theorem} 
\subsection{The angle function $\theta$}
We will assume from now on that an axial surface in $\SS ^2\times \RR$ has  axes  $Z$ and $Z^*$ that pass through
a fixed pair $O$ and $O^*$ of antipodal points in $\SS^2=\SS^2\times \{0\}$.  
Fix a stereographic projection from  $(\SS^2\times\{0\}) \setminus \{O^*\}$ to $\RR^2$, 
and let $\theta$ be the angle function on $(\SS^2\times\{0\}) \setminus \{O, O^*\}$ corresponding to the polar coordinate $\theta$ on $\RR^2$.   Extend $\theta$ to all of $(\SS^2\times \RR)\setminus (Z\cup Z^*)$ by requiring
that it be invariant under vertical translations.  Of course $\theta$ is only well-defined up 
to integer multiples of $2\pi$.


If $H$ is a helicoid with axes $Z$ and $Z^*$, we will call the components of $H\setminus(Z\cup Z^*)$
{\em half-helicoids}.   The half-helicoids are precisely the surfaces given by
\[
   \theta = \kappa\,  z + b.
\]
Here $2\pi/ \kappa$ is the pitch of the helicoid.  Rotating $H$ by an angle $\beta$ changes the corresponding
$b$ to $b+\beta$.  
Note that the entire helicoid $H$ consists of $Z\cup Z^*$ (where $\theta$ is not defined) together with all
points not in $Z\cup Z^*$ such that
\[
  \text{$\theta \cong \kappa\, z + b$  (mod  $\pi$)}.
\]

\subsection{ The restriction of $\theta$ to an annular slice}
Let $I\subset \RR$ be a closed interval (possibly infinite) and let
\[
   E = M \cap (\SS^2\times I)
\]
be the portion of $M$ in $\SS^2\times I$.  Suppose that $E$ is an annulus.
Then $E \setminus(Z\cup Z^*)$ consists of two simply connected domains that are congruent
by the involution $\rho_Z$.  Denote by $D$ one of these domains, and consider the restriction
 of $\theta$ to $D$.  Because $D$ is simply connected, we may choose a single-valued branch of this function, and we will  also refer to it as $\theta$ when there is no ambiguity.   Note that $\theta$
 extends continuously from $E$ to $\overline{E}$ since $\overline{E}$ has a well-defined
 tangent halfplane at all points of $\overline{E}\setminus E \subset Z\cup Z^*$.

\begin{definition}\label{alpha-beta-phi-def}
\begin{align*}
\alpha(z) &=\max \{ \theta (p,z)\, : \, (p,z)\in \overline{D} \},  \\
\beta(z) &= \min  \{ \theta (p,z)\, : \, (p,z)\in \overline{D} \}, \\
\phi(z) &= \alpha(z)-\beta(z).
\end{align*}
\end{definition}

Note that $\alpha(z)=\beta(z)$ if and only if $D\cap (\SS^2\times \{z\})$ is half of a great circle.
Note also that  $E$ is a portion of a helicoid if and only if
\[
  \alpha(z) \equiv \beta(z) \equiv \kappa\, z+b
\]
for some $\kappa$ and $b$. 

\begin{lemma}\label{alphabeta}
 The  functions $\alpha$, $-\beta$,
 and $\phi=\alpha-\beta$ are  convex,
and they are strictly convex unless  $E$ is contained in a helicoid.
\end{lemma}

\begin{proof}
Suppose that $\alpha$ is not  strictly convex. Then there exists $z_0<z_1<z_2$ 
such that 
\[ 
       \alpha (z_1)\geq \lambda \alpha (z_0) +(1-\lambda) \alpha(z_2),
\]
where $z_1 =\lambda z_0+ (1-\lambda) z_2$ and  $0<\lambda<1$.
 Let 
 \[    
      \kappa = \frac{\alpha (z_2)-\alpha (z_0)}{z_2-z_0},
\]
 and choose
$b$ to be the smallest value so that $\kappa \, z+b\geq \alpha(z)$ for all $z$ in $[z_0, z_2]$. 
 It follows that there is a value of
  $z$, say $z^*$, in the interior of the interval $[z_0,z_2]$  for which 
               $\alpha(z^*) =  \kappa \,  z^*+b$.  
 Let $p$ be a point in $\overline{D}\cap (\SS^2\times\{z^*\})$ with
 \[
   \theta(p) = \alpha(z^*).
\]
Then in a neighborhood of $p$, the surface $D$ lies on one side of the half-helicoid $H$ given by 
        $\theta = \kappa \,  z+b$, and the two surfaces
touch at $p$.  By the maximum principle
     (or the boundary maximum principle if $p$ is boundary point of $D$) together with analytic continuation,            
     $D\subset H$.  
 
 The  statements about convexity and strict convexity of $-\beta$ (or, equivalently, about concavity
 and strictly concavity of $\beta$) are proved in exactly the same way.
 
 The assertions about $\alpha-\beta$ follow, since the sum of two convex functions is convex, 
 and the sum is strictly convex if either summand is strictly convex.
\end{proof}


\section{The proof of Theorem~\ref{maintheorem}}

Consider first an annular end $E$.  We may suppose that $E$ is properly embedded in 
   $\SS^2\times [a,\infty)$.
 Choose $z_n\to \infty$ such that
  \[
            c:= \limsup_{z\to \infty} \phi(z)= \lim_{z_n\rightarrow \infty}\phi(z_n)
  \]
  where $\phi=\alpha-\beta$ is as in Definition~\ref{alpha-beta-phi-def}.
  Let $E_n$ and $D_n$ be the result of translating $E$ and $D$ downard by $z_n$.
 Since we are assuming that the curvature of $E$ is bounded, we may assume by passing
 to a subsequence that the $E_n$ converge smoothly to a properly embedded minimal annulus $E^*$,
 and that the $D_n$ converge smoothly to $D^*$, one
  of the connected components of $E^*\setminus (Z\cup Z^*)$.
  The smooth convergence $D_n\to D^*$ implies that the functions $\phi_n(z)=\phi(z-z_n)$  converge smoothly to 
  the corresponding angle difference function $\phi^*$ coming from $D^*$.
  Thus $\phi^*(z)$ attains its maximum value of $c$ at $z=0$.
  Consequently, $\phi^*$ is not strictly convex, so by Lemma~\ref{alphabeta},  $D^*$ is contained in a helicoid, and therefore
$\phi^*\equiv 0$.  In particular, $c=0$.

Returning our attention to the original surface $D$, we have $\phi =\alpha -\beta >0$, and since $\alpha$ is convex and $\beta$ is concave, there is a line that lies in the region below the graph of $\alpha$ 
and above the graph of $\beta$. 
Since $\alpha(z)-\beta(z)\to c = 0$ as $z\to\infty$, there is 
a unique line, say  the graph of $\theta = \kappa \, z + b$, that lies between the graphs of $\alpha$ and $\beta$,
and these graphs are asymptotic to this line.  Thus $D$ is $C^0$-asymptotic to the half-helicoid  whose equation is
     $\theta= \kappa \, z + b$. 
Since the curvature of $D$ is bounded, the surface $D$ is smoothly asymptotic to that half-helicoid.  
It follows immediately that the end 
\[
   E = \overline{D} \cup \rho_Z \overline{D}
\]
is asymptotic to the corresponding helicoid. This proves statement~(1) of Theorem~\ref{maintheorem}.

To prove statement~(2), suppose that $M$ is a properly embedded, axial minimal annulus.
Let $D$ be one of the simply connected components of $M\setminus (Z\cup Z^*)$.  We know from
Lemma~\ref{alphabeta} that $\phi$ is convex on all of $\RR$, and from the proof above of the first statement  of Theorem~\ref{maintheorem} (applied to the ends of $M$) that  
\[
  \lim_{z\to \pm \infty} \phi(z) = 0.
\]
Thus $\phi(z)\equiv 0$, so by Lemma~\ref{alphabeta}, $M$ is a helicoid.

Statement~(3) of Theorem~\ref{maintheorem} follows from a standard flux argument
as follows.
Let $s<t$ and let 
\[
  \Sigma=\Sigma(s,t)=M\cap (\SS^2\times (s,t)).
\]
Let $\nu(p)$ be the outward unit normal at
$p\in \partial \Sigma$. 
Since $\partial/\partial \theta$ is a Killing field on $\SS^2\times \RR$,
\[
  \int_{\partial \Sigma} \left( \nu\cdot \frac{\partial}{\partial\theta} \right)\, ds = 0
\]
by the first variation formula.
Equivalently, if we let $M_a = M\cap \{z\le a\}$, then the flux
\[
   \int_{\partial M_a} \left(\nu \cdot  \frac{\partial}{\partial\theta}\right)\,ds \tag{*}
\]
is independent of $a$.  We call \thetag{*} the {\em rotational flux} 
of $M$ (with respect to the axes $Z$ and $Z^*$).

If $M$ is asymptotic (as $z\to \infty$ or as $z\to -\infty$) to a helicoid $H$, then $M$ and $H$
clearly have the same rotational flux.   Thus to prove statement~(3), it suffices to check that if two helicoids with axes
$Z \cup Z^*$ have the same rotational flux, then they are congruent by rotation.
If we let $F(\kappa)$ denote the rotational flux of the helicoid $H(\kappa)$ given by
\[
        \text{ $ \theta \cong  \kappa \, z$ (mod $\pi$),}
  \]
then it suffices to show that $F(\kappa)$ depends strictly monotonically on $\kappa$. 
To see it does, note that in the expression
\[
  F(\kappa)=  \int_{ \partial (H(\kappa)\cap \{z\le 0\}) } \left(\nu \cdot  \frac{\partial}{\partial \theta} \right) \, ds,
\]
the integrand at each point is a strictly increasing function of $\kappa$ (except at the two points 
$O$ and $O^*$), and thus that $F(\kappa)$ is a strictly increasing function of $\kappa$.
 (At each point of 
  $\SS^2\times \{0\}$ other than $O$ and $O^*$, the larger $\kappa$ is, the smaller the angle between the vectors
  $\partial / \partial \theta$ and $\nu$.) \qed

\begin{remark}
The reader may wonder why we used rotational flux rather than the vertical flux
\[
    \int_{\partial M_a} \left( \nu \cdot \frac{\partial}{\partial z} \right) \, ds.
\]
The problem  with vertical flux is that the  helicoid $H(\kappa)$ and its mirror image $H(-\kappa)$ 
 have the same
vertical flux (equal to $\frac{2\pi r}{\sqrt{1+\kappa^2}}$).   Thus vertical flux alone does not rule out the possibility that the two 
ends of $M$ might be asymptotic to helicoids that are mirror images of each other.
\end{remark}
 

\begin{remark}\label{displaced-ends}
We have not proved that the constant terms $b$ in the equations 
\[
          \theta \cong \kappa \,  z + b \quad\text{(mod $\pi$)}
\]
for the helicoids asymptotic to the ends of $M$ are the same. 
There is some reason to expect that $b$ can change from end to end.

 A change in $b$ corresponds to a rotation, and  when $\kappa \neq 0$ (i.e. when the helicoid is not a cylinder) a rotation by $\beta$ is equivalent to a translation by $\beta/ \kappa$.    In the Introduction, we discussed known examples of properly embedded axial minimal surfaces of finite genus.  Those examples may be regarded as
desingularizing the intersection of a helicoid $H$ with the totally geodesic sphere $\SS^2\times\{0\}$. 
Such desingularization might well cause a slight vertical separation  of the top and bottom ends of the helicoid, in order to ``make room'' for the sphere. A similar situation exists in $\RR ^3$ when considering the Costa-Hoffman-Meeks surfaces as desingularizations of the intersection of a vertical catenoid with a horizontal plane passing through the waist of the helicoid \cite{HoffmanMeeks1990}, \cite{HoffmanKarcher1997}. 
While the top and bottom catenoidal ends have the same logarithmic growth rate, corresponding to the vertical flux,  numerical evidence from computer simulation of these surfaces indicates a vertical separation of the those ends.  
 (In other words, the top end is asymptotic to the top of a catenoid, the bottom end is asymptotic to the bottom of a catenoid, and numerical evidence indicates that the two catenoids are related by a nonzero vertical translation.)
\end{remark}

\newcommand{\hide}[1]{}

\end{document}